\documentclass[12pt]{amsart}
\usepackage{xypic}
\usepackage{amsmath}
\usepackage{amssymb}
\usepackage{enumerate}
\usepackage{color}
\usepackage{pgf}
\usepackage{tikz}
\usetikzlibrary{arrows,automata}
\usepackage{relsize}
\usepackage{latexsym}
\usepackage{pgfplots}
\usetikzlibrary{positioning}
\usepackage[T1]{fontenc}
\usepackage{soul} 
\usepackage{amsfonts}
\usepackage[T1]{fontenc}
\usepackage[latin9]{inputenc}
\usepackage{geometry}
\usepackage{amsmath}
\usepackage{amsthm}
\usepackage{stmaryrd}
\usepackage{stackrel}
\usepackage[all]{xy}
\theoremstyle{remark}

\makeatletter
\theoremstyle{plain}
\newtheorem{thm}{\protect\theoremname}[section]
\theoremstyle{definition}
\newtheorem{example}[thm]{\protect\examplename}
\theoremstyle{definition}
\newtheorem{defn}[thm]{\protect\definitionname}
\theoremstyle{plain}
\newtheorem{lem}[thm]{\protect\lemmaname}
\ifx\proof\undefined
\newenvironment{proof}[1][\protect\proofname]{\par
	\normalfont\topsep6\p@\@plus6\p@\relax
	\trivlist
	\itemindent\parindent
	\item[\hskip\labelsep\scshape #1]\ignorespaces
}{%
	\endtrivlist\@endpefalse
}
\providecommand{\proofname}{Proof}
\fi
\theoremstyle{plain}
\newtheorem{prop}[thm]{\protect\propositionname}
\theoremstyle{remark}

\date{}

\makeatother

\usepackage{babel}
\providecommand{\definitionname}{Definition}
\providecommand{\examplename}{Example}
\providecommand{\lemmaname}{Lemma}
\providecommand{\propositionname}{Proposition}
\providecommand{\remarkname}{Remark}
\providecommand{\theoremname}{Theorem}

\begin{document}
\title{A lattice-ordered monoid on multilayer networks}

\author[J. D\'iaz-Boils]{J. D\'iaz-Boils}
\address{Departament d'Economia Aplicada\\
Facultat d'Economia\\
Avinguda dels Tarongers\\
Universitat de Val\`encia\\
46022-Val\`encia. Spain.}
\email{joaquin.diaz@uv.es}

\author[O. Galdames-Bravo]{O. Galdames-Bravo}

\address{Departament de Matem\`atiques\\ 
CIPFP Vicente Blasco Ib\'a\~{n}ez\\
Gran Via del Regne de Val\`encia, 46\\
46005-Val\`encia. Spain.}

\email{galdames@uv.es}


\maketitle

\begin{abstract}
{In the present paper we introduce a lattice-ordered
partial monoid structure on a suitable set of multilayer networks.
We first study a kind of mappings that preserve the partial 
order and describe the order structure. After that we define 
the lattice-ordered monoid and deduce the main properties.}
{lattice-ordered monoid, multilayer network, interior mapping, partial operation.}
\\
2020 Math Subject Classification: Primary 06A06, Secondary 05C99
\end{abstract}

\section{Introduction}\label{intro}

On the one hand a multilayer network can be seen as a graph or a multigraph of 
graphs structures and they are habitually used as a tool for 
the study in applied science by means of mathematical 
formulations evolving for instance graph theory, topology or 
statistics, see for instance \cite{Bia,Bocetal} and references therein.
On the other hand lattice ordered monoids \cite{Birk} has been widely studied 
from several points of view (see e.g. \cite{Keh,Yoe} and references therein).
In the present paper we propose a join scheme of both conceps, multilayer network 
and lattice ordered monoid.\\

Our original interest on such structures is due to the fact 
that they provide an algebraic framework for an abstract notion of \emph{embodiment} in Neuroscience
by means of multilayer networks with a partial structure developed by the first author in \cite{Sign}.
This structure opens the possibility to a dynamical behaviour, which 
needs a suitable setting for being studied. At this point we obviate the 
classical interaction of an static network and focus on the algebraic 
structure that we define and how it can change the network structure.
The ideas we develop are mainly oriented to the original example 
described in \cite{Sign}, but we notice that one can easily extrapolate it to 
any other contexts where it appear multilayer networks or related structures
as, for example, multiplex networks, general networks or simply graphs and multigraphs.
We also notice the structure we define is actually a 
partial commutative monoid for our convenience, 
but the theory we develope apply to general 
commutative monoids. As far as we know there is not in the literature such an approach for a partial operation.\\

We outline the paper as follows. In Section \ref{intro} we introduce the paper 
and present the example that inspire us, then we depict in Section \ref{pomonoid} the first properties
for the partial ordering we introduce for multilayer networks and define 
a special sort of mappings. The aim of Section \ref{interior} 
is the study such mappings as interior mappings. 
Section \ref{lattice} is devoted to obtain
some results by applying the lattice structure we endow to our set 
of multilayer networks. Finally in Section \ref{monoid} we 
deduce key properties for the lattice-ordered monoid that we 
suggest for our scheme.\\

Let us define the set of multilayer networks we will deal.
Let $X$ be a set, a multisubset is a pair $(Y,m)$ where $Y$ is the underlying 
subset of $X$ and $m\colon Y\to\mathbb{Z}^+$ is the \emph{multiplicity function} that 
assigns to each element in $Y$ the number of occurrences (see \cite{Bliz}). 
Next definitions can be found in \cite{Sign} which are based in \cite{Baezetal} 
and \cite{McLane}, where multilayer networks are included into an abstraction 
called network model. There are other definitions by means the classical 
adjacency tensor representation (see for instance \cite{Dometal,Kivetal} and 
references therein). In order to simplify the abstract definition of network 
model and not lose the tensor representation we propose a definition which is 
halfway between both definitions.\\ 

A \emph{multigraph} $G$ on a set of \emph{nodes} $V(G)$ is a multisubset 
of \emph{edges} $E(G)$ that corresponds to pairs of elements of $V(G)$, 
together with the multiplicity function $m_G\colon E(G)\rightarrow\mathbb{Z}^+$. 
Similarly, the edges could have diferent colors. Let $C$ be a finite set of colors, 
and $col_G\colon E(G)\to \mathcal{P}(C)$ a mapping that assigns to each edge a 
subset of colors. Then, a \emph{layer} is the pair $(G,col_G)$, where $G$ is a 
multigraph. We will identify the layer $(G,col_G)$ with $G$ and, for $s\in\mathbb{N}$, 
we say that a layer is \emph{$s$-colored} if $col_G$ is onto and $s=|C|$, i.e.
$s$ denotes the number of colors included into the layer.
Let the set of nodes indexed by the set $\{1,\dots,n\}$ and denote 
by $MG(n)$ the set of multigraphs with such $n$ nodes. Let $c$ be 
a single color, then we denote by $MG^{c}(n)$ the set of $1$-colored layers.
Let $C=\{c_1,\dots,c_m\}$ be a set of colors, then we define the 
set of multilayer networks as the product
$$
MG^{\otimes C}(n):=(MG^{c_1}\otimes\cdots\otimes MG^{c_m})(n)
= MG^{c_1}(n)\times\cdots\times MG^{c_m}(n)\,.
$$
So every multigraph in $MG^{\otimes C}(n)$ is called a \emph{$|C|$-colored 
multilayer}. We observe that the tensor product represents the way in which the 
different layers of the multilayer are presented, taking care of the ordering. The details of these 
definitions are given in \cite{Sign}.
We will fix the nodes to the finite set $V$, so we just denote such a set by $MG^{\otimes C}$.
Now we are in position to define a commutative binary operation in $MG^{\otimes C}$. Let us denote by $\sqcup$ the disjont union of sets.
\begin{defn}
Let $C$ and $V$ be fixed sets of colors and nodes respectively.
Let a $s$-colored layer $G\in MG^{\otimes C}$ and a $q$-colored layer 
$H\in MG^{\otimes C}$, and assume that $C_1:=col(E(G))\subseteq C$, 
$C_2:=col(E(H))\subseteq C$ and that $V(G),V(H)\subseteq V$. Then the operation
\[
\odot\colon MG^{\otimes C}(n)\times MG^{\otimes C}(m)\longrightarrow MG^{\otimes C}
\] 
produces a new ($s+q-r$)-colored layer $G \odot H$, where $r=|C_1\cap C_2|$ with $n+m-p$ vertices 
where $p=\left|V(G)\cap V(H)\right|$ defined as 
$V(G\odot H):=V(G)\sqcup V(H)$, $E(G \odot H):=E(G) \cup E(H)$,
$m_{G \odot H}:=m_G + m_H$ and $col_{G \odot H}:=col_G \cup col_H$, where
the mappings are defined by a natural way.
\end{defn}

We set $\odot$ to be a commutative operation and $\otimes$
not be and also establish that $\odot$ has priority over $\otimes$,
that is: 
\[
G\otimes H\odot K=G\otimes(H\odot K)
\]

Notice we have defined two different ways of composing: $\otimes$ and $\odot$. 
That is, we consider sets $MG^{\otimes C}$ of concatenations in the form 
$G_1\oslash^{1}\cdots\oslash^{k-1}G_k$ with $\oslash^{i}\in\{\otimes,\odot\}$ 
for $\left|C\right|=k$ and $i=1,\dots,k-1$.
Also notice that, with this notation, we have obviated the interactions between 
layers which are present by the tensor product, but not explicitly: we just 
take into account the case when the relation between layers dissapear by means 
of the composition operation $\odot$.

\begin{example}\label{ex1}
For $k=3$ we have the concatenations 
\[
\begin{array}{c}
MG^{\otimes C} :=\{G\otimes H\otimes K,G\otimes K\otimes H,H\otimes G\otimes K,H\otimes K\otimes G,K\otimes G\otimes H,K\otimes H\otimes G,\\
G\odot H\otimes K,G\odot K\otimes H,H\odot K\otimes G,G\otimes H\odot K,H\otimes G\odot K,K\otimes G\odot H,G\odot H\odot K\}
\end{array}
\]
\end{example}

The following example illustrates the composition operation $\odot$:
\begin{example}
For $n=3,m=4,s=q=2$ and $p=3$:

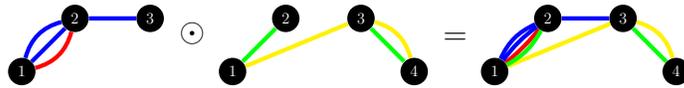
\vspace{2em}
\hspace{6em}
\begin{tikzpicture}[auto,node distance=2cm,scale=0.50, transform shape]  
\begin{scope}
\tikzstyle{every state}=[fill=black,draw=none,text=white,minimum size = 2pt]
  \node[state]         (A)                    {$1$};   
  \node[state]         (B) [above right of=A] {$2$};       
  \node[state]         (C) [right of=B]       {$3$};   
\draw[line width=1.6pt] (A) edge[color=blue, bend left] (B); 
\draw[line width=1.6pt] (A) edge[color=blue] (B); 
\draw[line width=1.6pt] (A) edge[color=red,bend right] (B);
\draw[line width=1.6pt] (B) edge[color=blue] (C);
\end{scope}
\end{tikzpicture}
\begin{tikzpicture}[auto,node distance=2cm,thick,scale=0.50, transform shape]
\begin{scope}

\tikzstyle{every state}=[fill=black,draw=none,text=white,minimum size = 2pt]
  \node[state]         (A)                    {$1$}; 
  \node[state]         (B) [above right of=A] {$2$};       
  \node[state]         (C) [right of=B]       {$3$};
  \node[state]         (D) [below right of=C]       {$4$};  
\draw[line width=1.6pt] (A) edge[color=green]            (B); 
\draw[line width=1.6pt] (A) edge[color=yellow] (C); 
\draw[line width=1.6pt] (C) edge[color=yellow,bend left]  (D);
\draw[line width=1.6pt] (C) edge[color=green]            (D);
\end{scope}
\node[above left=2em]{$\mathlarger{\mathlarger{\mathlarger{\mathlarger{\mathlarger{\odot}}}}}$};
\end{tikzpicture}
\begin{tikzpicture}[auto,node distance=2cm,scale=0.50, transform shape]  
\begin{scope}
\tikzstyle{every state}=[fill=black,draw=none,text=white,minimum size = 2pt]
  \node[state]         (A)                    {$1$}; 
  \node[state]         (B) [above right of=A] {$2$};       
  \node[state]         (C) [right of=B]       {$3$};
  \node[state]         (D) [below right of=C]       {$4$};  
\draw[line width=1.6pt] (A) edge[color=red, bend right=6]            (B); 
\draw[line width=1.6pt] (A) edge[color=yellow] (C); 
\draw[line width=1.6pt] (C) edge[color=yellow,bend left]  (D);
\draw[line width=1.6pt] (C) edge[color=green]            (D);
\draw[line width=1.6pt] (A) edge[color=blue, bend left=20] (B); 
\draw[line width=1.6pt] (A) edge[color=blue, bend left=8] (B); 
\draw[line width=1.6pt] (A) edge[color=green,bend right=16] (B);
\draw[line width=1.6pt] (B) edge[color=blue] (C);
\end{scope}
\node[above left=2em]{$\mathlarger{\mathlarger{\mathlarger{\mathlarger{\mathlarger{=}}}}}$};
\end{tikzpicture}

\vspace{1em}
\hspace{-1.5em}Note that new colors appear in a layer after more applications of $\odot$.
\end{example}

\section{The partial ordered structure}\label{pomonoid}

Operation $\odot$ defined in previous section can be seen as 
an accumulation of vertices and edges of two given layers that 
becomes a new layer with more colors than the original ones. 
For example, given the multilayers 
$G\otimes H\otimes K,G\odot H\otimes K\in MG^{\otimes C}$,
we understand that $G\odot H\otimes K$ is, in some sense, 
over or below from $G\otimes H\otimes K$. By convention we 
say that $G\otimes H\otimes K\le G\odot H\otimes K$, since
we consider that $G\odot H$ is more complex, in some sense, 
than $G\otimes H$. Let us formalize this idea.\\

A partially ordered set or a \emph{poset} is a set with a binary operation
$\le$ wich is reflexive, antisymmetric and transitive (see e.g. \cite{Birk}).
We define the relation $\le$ in $MG^{\otimes C}$ by ordering the concatenations of 
multigraphs as given in the following. Let $k=|C|$ for the rest of the section.
\begin{defn}
Given $G_{1}\oslash^{1}\cdots\oslash^{k-1}G_{k}$ and $G_{1}\ominus^{1}\cdots\ominus^{k-1}G_{k}$
in $MG{}^{\otimes C}$
with $\oslash^{i},\ominus^{i}\in\{\otimes,\odot\}$ for $i=1,\dots,k-1$
we write
\[
G_{1}\oslash^{1}\cdots\oslash^{k-1}G_{k}\le G_{1}\ominus^{1}\cdots\ominus^{k-1}G_{k}
\]
if and only if there is no $i\in\{1,\dots,k-1\}$ such that $\oslash^{i}=\otimes$
and $\ominus^{i}=\odot$.
\end{defn}

This partial order allows us to define the following mappings. In order to simplify 
the notation, we sometimes will use lowercase letters as multilayers of $MG^{\otimes C}$.
\begin{defn}
Let the mapping $f_{j}\colon MG^{\otimes C}\to MG^{\otimes C}$:
\[
f_{j}(x)=\begin{cases}
G_{1}\oslash^{1}\cdots G_{j}\odot G_{j+1}\cdots\oslash^{k-1}G_{k} & 
\textrm{if }x=G_{1}\oslash^{1}\cdots G_{j}\otimes G_{j+1}\cdots\oslash^{k-1}G_{k}\\
x & \textrm{otherwise}
\end{cases}
\]
for $j=1,\dots,k-1$. We say that 
$x,y\in MG^{\otimes C}$ are \emph{comparable through $f_j$} if $f_j(x)=y$.
\end{defn}

By adding $f_0$ as the identity, it is easy to see that $f_{j}$ are order-preserving.
For the sake of clarity we use the notation $f_{j}$ for any mapping defined above,
avoiding the list of indexes. These mappings will be useful in the sequel, the next example illustrates how these functions 
work and describe, in some sense, a flow on $MG^{\otimes C}$ as a poset.

\begin{example}\label{ex2}
For the elements in Example \ref{ex1} we have: 
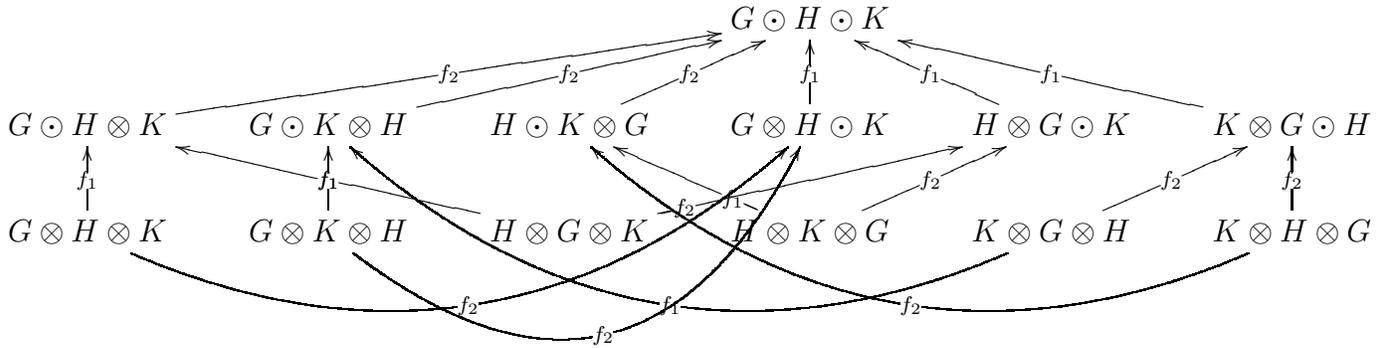
{
\[
\xymatrix{ &  & & G\odot H\odot K\\
G\odot H\otimes K\ar[urrr]|{f_{2}} & G\odot K\otimes H\ar[urr]|{f_{2}} & 
H\odot K\otimes G\ar[ur]|{f_{2}} & G\otimes H\odot K\ar[u]|{f_{1}} & 
H\otimes G\odot K\ar[ul]|{f_{1}} & K\otimes G\odot H\ar[ull]|{f_{1}}\\
G\otimes H\otimes K\ar[u]|{f_{1}}\ar@/_{4pc}/[urrr]|{f_{2}} & 
G\otimes K\otimes H\ar[u]|{f_{1}}\ar@/_{5pc}/[urr]|{f_{2}} & 
H\otimes G\otimes K\ar[ull]|{f_{1}}\ar[urr]|(.24){f_{2}} & 
H\otimes K\otimes G\ar[ul]|(.32){f_{1}}\ar[ur]|{f_{2}} & 
K\otimes G\otimes H\ar@/^{4pc}/[ulll]|{f_{1}}\ar[ur]|{f_{2}} & 
K\otimes H\otimes G\ar@/^{4pc}/[ulll]|{f_{2}}\ar[u]|{f_{2}}
}
\]
}

From the example above we extract two immediate results.
The first one establishes that one can obtain the top element after
an action of every $f_{j}$ over a given concatenation whatever ordering could
be and the second that $f_{j}$ are increasing.
\end{example}
\begin{prop}
$f_{i_{1}}\cdots f_{i_{k}}(G_{1}\oslash^{1}\cdots\oslash^{k-1}G_{k})=G_{1}\odot\cdots\odot G_{k}$
for $i_{1}<\cdots<i_{k}$ a permutation of $1,\dots,k$.
\end{prop}
\begin{prop}
$f_{j}(G_{1}\oslash^{1}\cdots\oslash^{k-1}G_{k})\geq G_{1}\oslash^{1}\cdots\oslash^{k-1}G_{k}$.
\end{prop}

To the aim of simplicity we will focus our study to a fixed set of 
multilayers/multigraphs. Let us fix a list of multigraphs 
$G(k):=(G_1,\dots,G_k) \in (MG^{\otimes C})^k$ and denote 
$$
\bigcirc G(k):=\{G_{1}\oslash^{1}\cdots\oslash^{k-1}G_{k}:
\oslash^{i}\in\{\otimes,\odot\}\}\,.
$$
Notice that $MG^{\otimes C}=\bigcup_{k=|C|} \{\bigcirc G(k):G(k)\in (MG^{\otimes C})^k\}$ and moreover that such subsets of $MG^{\otimes C}$ are invariant by $f_j$, i.e. 
$f_j(\bigcirc G(k))\subseteq\bigcirc G(k)$. Hence, from these comments
we deduce that $f_j|_{\bigcirc G(k)}\colon\bigcirc G(k)\to\bigcirc G(k)$
is well defined and from now on we understand $f_j$ as $f_j|_{\bigcirc G(k)}$
for some $\bigcirc G(k)$.\\

Let $P$ be a poset. We say that $b \in P$ is a \emph{bottom} element if 
$b\le x$ for every $x \in P$ and $a\in P$ is a \emph{top} element
if $a\ge x$ for every $x\in P$ (see \cite{Birk}). 
\begin{lem}\label{lem1}
$\bigcirc{G(k)}$ is a partial ordered set with top $G_{1}\odot\cdots\odot G_{k}$.
\end{lem}
\begin{proof}
Observe the order of $\bigcirc G(k)$ is described by the mappings $f_j$ 
(see Example \ref{ex2}). Reflexivity is given by $f_0$ while transitivity is 
immediate by definition of the mappings $f_j$. For antisymmetry 
we recall the form of the ordering given in the previous definition, 
now a concatenation can only be compared both ways with another concatenation if they 
are both the same. In that case they are compared by means of the 
same $f_j$ whenever a $\odot$ appears in the $j$-position of the concatenation. 
\end{proof}

\begin{example}
The following diagram illustrates the argument used for the antisymmetry in the proof above:
$$
\xymatrix{G\odot H\otimes K\ar@/^{1pc}/[r]^{f_{1}} & G\odot H\otimes K\ar@/^{1pc}/[l]^{f_{1}}}
$$
while $G\otimes H\odot K$ and $G\odot H\otimes K$  are not comparable through any mapping $f_j$.
\end{example}
Notice that we cannot dualize the above since inverse mappings in  such as $g_{1}$ for which 
\[
g_{1}(G\odot H\otimes K)=G\otimes H\otimes K
\]
lose the well-definedness condition for the non commutativity of $\otimes$.\\

We now prove a notable property that we develop in section below.
\begin{defn}
A \emph{closure mapping} on a poset $P$
is a monotone map $g:P\rightarrow P$ that is 
\begin{enumerate}
\item increasing, i.e. for all $x\in P,gx\geq x$ and 
\item idempotent, i.e. for all $x\in P,g^{2}x=gx$.
\end{enumerate}
\end{defn}
\begin{prop}
The mappings $f_{j}$ are closure mappings.
\end{prop}

The mapping determined by two elements is defined in \cite{Cha} as
\[
f_{a,b}(x)=\left\{ \begin{array}{lll}
b & \mbox{ if }x=a \\
x & \mbox{ otherwise.}
\end{array}\right.
\]
Let us see by an example that these mappings are closely related to 
our mappings $f_j$. If we change elements by tuples we obtain the following example.
\begin{example}\label{ex3}
Let $\vec{a}_j=
(G_1\oslash^1\cdots\oslash^{j-1} G_j\otimes G_{j+1}\oslash^{j+1}\cdots\oslash^{k-1}
 G_k)_{\oslash^i\in\{\otimes,\odot\}}$
and $\vec{b}_j=
(G_1\oslash^1\cdots\oslash^{j-1} G_j\odot G_{j+1}\oslash^{j+1}\cdots\oslash^{k-1}  
G_k)_{\oslash^i\in\{\otimes,\odot\}}$, where the tuples run all the combinations
of $\oslash^i\in\{\otimes,\odot\}$ and  $i$ runs the set $\{1,\dots,k-1\}\setminus\{j\}$,
taking into account that $\odot$ is commutative and $\otimes$ is not commutative.\\

For instance, we get $k=6$, $j=2$ and fix the multigraphs $G_1,\dots,G_6$ all differents. Then 
the set of multilayers with the form $G_1\oslash G_2\otimes G_3\oslash G_4\oslash G_5\oslash G_6$
represent the tuple $\vec{a}$, namely
\begin{equation*}
\begin{split}
\vec{a}=&(G_1\otimes G_2\otimes G_3\otimes G_4\otimes G_5\otimes G_6,
 G_1\odot G_2\otimes G_3\otimes G_4\otimes G_5\otimes G_6,\\
&G_1\otimes G_2\otimes G_3\odot G_4\otimes G_5\otimes G_6,
 G_1\otimes G_2\otimes G_3\otimes G_4\odot G_5\otimes G_6,\\
&G_1\otimes G_2\otimes G_3\otimes G_4\otimes G_5\odot G_6,
 G_1\odot G_2\otimes G_3\odot G_4\otimes G_5\otimes G_6,\\
&G_1\odot G_2\otimes G_3\otimes G_4\odot G_5\otimes G_6,
 G_1\odot G_2\otimes G_3\otimes G_4\otimes G_5\odot G_6,\\
&G_1\otimes G_2\otimes G_3\odot G_4\odot G_5\otimes G_6,
 G_1\otimes G_2\otimes G_3\odot G_4\otimes G_5\odot G_6,\\
&G_1\otimes G_2\otimes G_3\otimes G_4\odot G_5\odot G_6,
 G_1\odot G_2\otimes G_3\odot G_4\odot G_5\otimes G_6,\\
&G_1\odot G_2\otimes G_3\odot G_4\otimes G_5\odot G_6,
 G_1\odot G_2\otimes G_3\otimes G_4\odot G_5\odot G_6,\\
&G_1\otimes G_2\otimes G_3\odot G_4\odot G_5\odot G_6,
 G_1\odot G_2\otimes G_3\odot G_4\odot G_5\odot G_6)\,.
\end{split}
\end{equation*}
And $G_1\oslash G_2\odot G_3\oslash G_4\oslash G_5\oslash G_6$ represents
the tuple $\vec{b}$, so $f_2 = f_{\vec{a},\vec{b}}$.
Observe that we must choose an order for the tuple. 
Also notice that all these elements are different, since
we have choosen all multigraphs different. Taking into account that $\vec{b}$ 
is the same tuple, just changing the second ``$\otimes$'' by 
``$\odot$'' in all entries.
\end{example}

Let us finish the section with two interpretations of the content 
defined so far that can be considered for further developments.
\begin{subsection}{Levels into $\bigcirc G(k)$}

Looking at Example \ref{ex2} we can organize $\bigcirc G(k)$
as a disjoint union of \emph{levels} according to the number of $\odot$
appearing in every concatenation. That is:
\[
\bigcirc G(k)=\bigsqcup_{0\leq l\leq k-1}{\bigcirc G(k)_{l}}
\]
where every $\bigcirc G(k)_{l}$ is the set of all concatenations with exactly
$l$ operators $\odot$ in it. In fact:
\[
f_{j}:\bigcirc G(k)_{l}\longrightarrow \bigcirc G(k)_{l+1}
\]
for which, when composing, we can jump more than one level in one
step by defining 
\[
f_{i}\circ f_{j}=f_{ij}:\bigcirc G(k)_{l}\longrightarrow \bigcirc G(k)_{l+2}
\]
which suggests considering mappings $f_{i_{1}\cdots i_{l}}$ where $i_{j}\in\{1,\dots,k-1\}$ for $1\leq j\leq l$ in the expected way for simultaneous
mergings, this allows to jump various levels at a time into the poset.\\
\end{subsection}
\begin{subsection}{Monads as modalities}
All the content introduced so far can be interpreted in terms of Category Theory as 
follows, where the terminology can be found for instance in \cite{Ben} and \cite{McLane}.
Considering posets as categories, it can be proved that $f_{j}$
are idempotent endofunctors and the fact that they are \emph{monads}. 
Then, one can see $f_{j}$ as (possibility) modalities $\diamondsuit_{j}$ for a 
certain Modal Logic system where we write
$\diamondsuit_{\alpha}=\diamondsuit_{i_{1}}\cdots\diamondsuit_{i_{l}}$
for $\alpha=i_{1}...i_{l}$ all distinct for idempotence. Now we have
a multimodal system where every modality is a conjunction of possible applications 
of functions $f_j$ satisfying the following axioms:
\begin{enumerate}
\item \label{ax1} $\diamondsuit_{\alpha}(x\wedge y)=x\wedge\diamondsuit_{\alpha}y$
\item \label{ax2} $\diamondsuit_{\alpha}(x\vee y)=x\vee\diamondsuit_{\alpha}y$
\end{enumerate}
We deduce the following easy property: 
$\diamondsuit_{\alpha}$ are \emph{strong functors} with the identity
as the strength, since Axiom \eqref{ax2} implies $\diamondsuit_{\alpha}(x\vee y)\geq\diamondsuit_{\alpha}x\wedge y$. 

\end{subsection}

\section{Interior mappings}\label{interior}

In section above we show that mappings $f_j$ can be seen as determined by two tuples. Despite our study is for interior mappings, as they are dual of closure mappings, results for $f_j$'s are easily deduced. Interior mappings have important properties for the analysis of posets as is shown in \cite{Yao}. In this section provide conditions for a mapping defined by two elements be interior mapping and so, conditions to be closure mappings.\\

The following definitions can be found in \cite{Cha}.
Let $(P,\le)$ be a poset. Let $A\subset P$, the sets
$\mathcal{L}(A):=\{x\in P:x\le A\}$ and $\mathcal{U}(A):=\{x\in P:x\ge A\}$
are respectively the \emph{lower} and \emph{upper cone} 
of $A$. Let the tuples $\vec{a}:=(a_{1},\dots,a_{n}),\vec{b}:=(b_{1},\dots,b_{n})\in P^{n}$
such that $a_{i}\ne b_{i}$ for $i=1,\dots,n$. The mapping determined
by such tuples is defined as 
\[
f_{\vec{a},\vec{b}}(x)=\left\{ \begin{array}{lll}
b_{i} & \mbox{ if }x=a_{i}\mbox{ for }i=1,\dots,n\\
x & \mbox{ otherwise.}
\end{array}\right.
\]
This mapping is strictly monotone if and only if $\vec{a}$
and $\vec{b}$ are not comparable, $\mathcal{L}(\{{\vec{a}}\})\setminus\{\vec{a}\}\subseteq\mathcal{L}(\{{\vec{b}}\})\setminus\{\vec{b}\}$
and $\mathcal{U}(\{{\vec{a}}\})\setminus\{\vec{a}\}\subseteq\mathcal{U}(\{{\vec{b}}\})\setminus\{\vec{b}\}$
(see \cite[Proposition 3.1]{Cha}). These conditions can be easily
changed to obtain a monotone mapping.\\

A mapping $f\colon P\to P$ is \emph{interior}
when for any $x,y\in P$: $f(x)\le y$ if and only if $f(x)\le f(y)$
(see \cite[Definition 3.1 and Remark 3.2]{Yao}). This definition
is equivalent to the following three axioms for $x,y\in P$: 
\begin{enumerate}
\item Monotonicity: $x\le y$ implies $f(x)\le f(y)$. 
\item Contraction: $f(x)\le x$. 
\item Idempotence: $f(f(x))=x$. 
\end{enumerate}
We say that $P$ is a \emph{bounded
poset} if it has bottom and top elements, it can also be \emph{lower bounded} and \emph{upper bounded}. The product poset $(P^{n},\le)$
is defined by means of the natural order and the fact that $P$ is
a bounded poset implies that $P^{n}$ is bounded again and the top
and bottom are $(a,\dots,a)$ and $(b,\dots,b)$ respectively,
where $a$ and $b$ are bottom and top elements of $P$. 
\begin{prop}
Let $P$ be a lower bounded poset. If $\vec{b}$ is a bottom
of $P^{n}$ and $f_{\vec{a},\vec{b}}$ is monotone, then the mapping
$f_{\vec{a},\vec{b}}$ is interior. 
\end{prop}
\begin{proof}
Notice that $\vec{b}=(b,\dots,b)$ for the bottom $b$ of
$P$. For the sake of clarity we denote $f:=f_{\vec{a},\vec{b}}$.
Therefore $f(x)=b$ or $f(x)=x$ for every $x\in P$. Assume $f(x)\le y$.
We have two cases: If $f(x)=b$, then $f(x)\le f(y)$, since $b$
is bottom in $P$. If $f(x)=x\le y$, then $f(x)\le f(y)$, since
$f$ is monotone. Now assume $f(x)\le f(y)$. If $f(x)=b$, then $f(x)\le y$
since $b$ is bottom in $P$. If $f(x)=x$ and $f(y)=b$, then $x\le b$,
so necessarily $x=b$. Thus, $f(x)=b\le y$ since $b$ is bottom.
If $f(x)=x$ and $f(y)=y$, then $f(x)\le y$ trivially. 
\end{proof}

Observe that the converse of proposition above is not true in general 
as is shown in the following example. 
\begin{example}
Let $\vec{a}_{j}=(G_{1}\oslash^{1}\cdots\oslash^{j-1}G_{j}\otimes G_{j+1}\oslash^{j+1}\cdots\oslash^{k-1}G_{k})_{\oslash^{i}\in\{\otimes,\odot\}}$
and $\vec{b}_{j}=(G_{1}\oslash^{1}\cdots\oslash^{j-1}G_{j}\odot G_{j+1}\oslash^{j+1}\cdots\oslash^{k-1}G_{k})_{\oslash^{i}\in\{\otimes,\odot\}}$,
where the tuples run all the combinations of $\oslash^{i}\in\{\otimes,\odot\}$
and $i$ runs the set $\{1,\dots,k-1\}\setminus\{j\}$, taking into
account that $\odot$ is commutative and $\otimes$ is not commutative.
Then the mapping $f_{\vec{a}_{j},\vec{b}_{j}}$ is interior, but $\vec{b}$
is not a bottom. What happens is that $b\preceq a$, i.e. $b$ is
covered by $a$ or in other words, there is no elements between $b$
and $a$, formally $b\le a$ and if $x\le a$, then $x\le b$. 
\end{example}
This idea allows us to obtain a better result. Previous lemma
could be useful in case we were not be able to find an element covered
by another. Observe that $\vec{b}\preceq\vec{a}$ if and only if $a_{i}\preceq b_{i}$
for $i=1,\dots,n$. We define the set of the tuple as $\{\vec{a}\}:=\{a_{1},\dots,a_{n}\}$. 
\begin{prop}
Let $P$ be a poset. If $\{\vec{b}\}\cap\{\vec{a}\}=\emptyset$ and
$\vec{b}\preceq\vec{a}$, then the mapping $f_{\vec{a},\vec{b}}$
is interior.
\end{prop}
\begin{proof}
Let us denote $f:=f_{\vec{a},\vec{b}}$. 

\begin{itemize}
\item Monotonicity: Assume $x\le y$ and $i,j\in\{1,\dots,n\}$
such that $i\ne j$. The case $x=y$ is clear from definition of mapping.
Having in mind that $\{\vec{b}\}\cap\{\vec{a}\}=\emptyset$, it is clear that
$f(b_{i})=b_{i}$ for every $i\in\{1,\dots,n\}$, we have the following
cases:
\begin{itemize}
\item If $x=a_{i}$ and $y=b_{j}$, then $f(x)=b_{i}$
and $f(y)=b_{j}$, by hypothesis $b_{i}\le a_{i}=x\le y=b_{j}$, hence
$f(x)\le f(y)$.\\
\item If $x=a_{i}$ and $y=a_{j}$, then $f(x)=b_{i}$ and $f(y)=b_{j}$,
by hypothesis $b_{i}\le a_{i}=x\le y=a_{j}$, but $b_{j}\preceq a_{j}$,
hence $a_{i}\le b_{j}$, so $f(x)\le f(y)$.\\
\item If $x\ne a_{i}$ and $y=a_{j}$, then $f(x)=x$ and $f(y)=b_{j}$,
by hypothesis $x\le y=a_{j}$, but $b_{j}\preceq a_{j}$, hence $x\le b_{j}$,
so $f(x)\le f(y)$.\\
\item The rest of cases brings us to $f(x)=x$ and $f(y)=y$, so
by hypothesis $f(x)\le f(y)$.\\
\end{itemize}
\item Contraction: If $x=a_{i}$ for some $i\in\{1,\dots,n\}$,
$f(x)=b_{i}\le a_{i}$. If $x\ne a_{i}$ for every $i\in\{1,\dots,n\}$,
$f(x)=x$. In both cases $f(x)\le x$.\\
\item Idempotence: If $x=a_{i}$ for some $i\in\{1,\dots,n\}$,
then $f(x)=b_{i}\preceq a_{i}$. Since $\{\vec{b}\}\cap\{\vec{a}\}=\emptyset$,
necessarily $f(b_{i})=b_{i}$, i.e. $f(f(x))=x$. If $x\ne a_{i}$
for all $i\in\{1,\dots,n\}$, is clear that $f(f(x))=x$. 
\end{itemize}
And the proof is ended.
\end{proof}

We can obtain, by duality, versions of propositions above for closure
mappings. We have omitted the proofs, since they are analog to the
ones above. 
\begin{prop}
Let $P$ be a upper bounded poset. If $\vec{b}$ is a top
of $P^{n}$ and $f_{\vec{a},\vec{b}}$ is monotone, then the mapping
$f_{\vec{a},\vec{b}}$ is closure. 
\end{prop}
\begin{prop}
Let $P$ be a poset. If $\{\vec{b}\}\cap\{\vec{a}\}=\emptyset$ and
$\vec{a}\preceq\vec{b}$, then the mapping $f_{\vec{a},\vec{b}}$
is closure.
\end{prop}

\section{The lattice structure}\label{lattice}

We saw above that we can define a partial order in a set of multilayer networks
and show that this order yields several properties in such a framework.
In this section we go a little further and provide a lattice structure for
$\bigcirc G(k)$. 
A meet (resp. join) semilattice is a poset $(L,\le)$ such
that any two elements $x$ and $y$ have a greatest lower bound (called
meet or infimum) (resp. a smallest upper bound (called join or supremum)),
denoted by $x\wedge y$ (resp. $x\vee y$).
A poset $(L,\le)$ is called a \emph{lattice} and denoted by 
$(L,\le,\wedge,\vee)$ if for every pair of elements we can construct 
into the lattice their meet and their join. These definitions can be 
found for instance in \cite{Birk}.
Let us define a meet and a join operators for the poset $(\bigcirc{G(k)},\leq)$:
\begin{defn}
Given $G_{1}\oslash^{1}\cdots\oslash^{k-1}G_{k}$, 
$G_{1}\ominus^{1}\cdots\ominus^{k-1}G_{k}\in \bigcirc{G(k)}$
(in short $\oslash G$ and $\ominus G$) we write $\oslash G\wedge\ominus G=\olessthan G$
for $G_{1}\olessthan^{1}\cdots\olessthan^{k-1}G_{k}$ such that 
\[
\olessthan^{j}=\begin{cases}
\otimes & \textrm{if }\oslash^{j}=\otimes\textrm{ or }\ominus^{j}=\otimes\\
\odot & \textrm{otherwise}
\end{cases}
\]
and we write $\oslash G\vee\ominus G=\ogreaterthan G$ for $G_{1}\ogreaterthan^{1}\cdots\ogreaterthan^{k-1}G_{k}$
such that 
\[
\ogreaterthan^{j}=\begin{cases}
\odot & \textrm{if }\oslash^{j}=\odot\textrm{ or }\ominus^{j}=\odot\\
\otimes & \textrm{otherwise}
\end{cases}
\]
\end{defn}

It can be easily checked the usual properties of both operations, that is: $x\wedge y\leq x,y$ and for every $z\leq x,y$
one has $z\leq x\wedge y$ and dually: $x\vee y\geq x,y$ and for
every $z\geq x,y$ one has $z\geq x\vee y$ for every $x,y,z\in \bigcirc{G(k)}$.

\begin{prop}
The \emph{absorption laws} are satisfied for every $x,y\in \mathbf{\bigcirc}{G(k)}$:
\begin{itemize}

\item{$x\vee (x\wedge{y})=x$}
\item{$x\wedge (x\vee{y})=x$}
\end{itemize}

\end{prop}
\begin{proof}
Let us prove the first assertion. For $x=\ominus G$ and $y=\obar G$ we construct $x\wedge{y}=\oslash G$ such that 
\[
\oslash^{j}=\begin{cases}
\otimes & \textrm{if }\obar^{j}=\otimes\textrm{ or }\ominus^{j}=\otimes\\
\odot & \textrm{otherwise}
\end{cases}
\] and $x\vee (x\wedge{y})=\ogreaterthan G$ as \[
\ogreaterthan^{j}=\begin{cases}
\odot & \textrm{if }\ominus^{j}=\odot\textrm{ or (}\obar^{j}=\odot\textrm{ and }\ominus^{j}=\odot)\\
\otimes & \textrm{otherwise}
\end{cases}
\] which can be expressed as 
\[
\begin{cases}
\odot & \textrm{if }\ominus^{j}=\odot\\
\otimes & \textrm{otherwise}
\end{cases}
\]
and becomes the same assignation considered for $x=\ominus G$.
\end{proof}

Let us recall that a \emph{minimal element} into a 
poset is an element such that it is not greater than any other element in the poset.
\begin{prop}
$(\bigcirc{G(k)},\leq,\wedge,\vee,1_{\odot},m_{\pi})$ is
an upper-bounded lattice where:
\begin{itemize}
\item $1_{\odot}=G_{1}\odot\cdots\odot G_{k}$ is the top and 
\item $m_{\pi}=G_{\pi(1)}\otimes\cdots\otimes G_{\pi(k)}$ are $k!$ 
minimal elements for $\pi$ a permutation of the set $\{1,\dots,k\}$.
\end{itemize}
\end{prop}
\begin{proof}
Check that $x\wedge1_{\odot}=x,x\vee1_{\odot}=1_{\odot}$
and $x\wedge m_{\pi}=m_{\pi},x\vee m_{\pi}=x$.
\end{proof}
\begin{prop}
$(\bigcirc{G(k)},\leq,\wedge,\vee,1_{\odot},m_{\pi})$ is
distributive.
\end{prop}
\begin{proof}
Let $x_{1}=\oslash G$, $x_{2}=\ominus G$, $x_{3}=\obar G$.
Now $x_{1}\wedge(x_{2}\vee x_{3})=\olessthan G$ where
\[
\olessthan^{j}=\begin{cases}
\otimes & \textrm{if }\ominus^{j}=\otimes\textrm{ or }\obar^{j}=\otimes\textrm{ and }\oslash^{j}=\otimes\\
\odot & \textrm{if }\ominus^{j}=\obar^{j}=\odot\textrm{ or }\oslash^{j}=\odot
\end{cases}
\]
which is exactly the same operator as
\[
\begin{cases}
\otimes & \textrm{if}\textrm{ no }(\oslash^{j}=\odot\textrm{ or }\ominus^{j}=\odot)\textrm{ or }\textrm{no }(\oslash^{j}=\odot\textrm{ or }\obar^{j}=\odot)\\
\odot & \textrm{otherwise}
\end{cases}
\]
for $(x_{1}\wedge x_{2})\vee(x_{1}\wedge x_{3})$.
\end{proof}
\begin{prop}\label{prop2}
Mappings $f_{j}$ preserve meets and joins.
\end{prop}
Following the notation of previous section we try some conditions in order to 
find mappings defined by two tuples that also preserve meets and joins.
Let $(L,\le,\wedge,\vee)$ be lattice and define the cartesian product
$(L^n,\le)$ and the meet and join operations defined coordinatewise for it, i.e.
for $\vec{a},\vec{b}\in L^n$ we define 
$\vec{a}\wedge\vec{b}:=(a_1\wedge b_1,\dots,a_n\wedge b_n)$ and
$\vec{a}\vee\vec{b}:=(a_1\vee b_1,\dots,a_n\vee b_n)$
from which one can easily verify the distributive properties.
We need the following property for $\vec{a}\in L^n$:
$$
x\ne \vec{a}\ne y \Longleftrightarrow x\wedge y\ne \vec{a}\,,
$$
that we say $\vec{a}$ is \emph{strictly not absorbing for $\wedge$}.
In an analogous way we define \emph{strictily not absorbing for $\vee$}.
In order to simplify the proof of the following proposition we 
have included the hypothesis $\{\vec{a}\}\cap\{\vec{b}\}=\emptyset$. 
\begin{prop}\label{prop1}
Let $L$ be a lattice. Assume that $\{\vec{a}\}\cap\{\vec{b}\}=\emptyset$.
\begin{enumerate}
\item\label{item1} If $\vec{b}$ is bottom element of $L^n$ and 
$\vec{a}$ is strictily not absorbing for $\wedge$, then mappings 
$f_{\vec{a},\vec{b}}$ preserve meets. 
\item\label{item2} If $\vec{b}$ is top element of $L^n$ and
$\vec{a}$ is strictily not absorbing for $\vee$, then mappings 
$f_{\vec{a},\vec{b}}$ preserve joins. 
\end{enumerate}
\end{prop}
\begin{proof}
\eqref{item1} Assume $\vec{b}$ is bottom, 
then $\vec{b}\wedge \vec{b} = x \wedge \vec{b} =\vec{b}\wedge y =
\vec{b}$. Observe that $f_{\vec{a},\vec{b}}(x\wedge y)\in\{\vec{b},x\wedge y\}$.
Also $f_{\vec{a},\vec{b}}(x)\in\{\vec{b},x\}$ and
$f_{\vec{a},\vec{b}}(y)\in\{\vec{b},y\}$, thus
$f_{\vec{a},\vec{b}}(x)\wedge f_{\vec{a},\vec{b}}(y)\in\{\vec{b},x\wedge y\}$.
As $\vec{b}$ is bottom $f_{\vec{a},\vec{b}}(x)\wedge f_{\vec{a},\vec{b}}(y)=x\wedge y$
if and only if $x\ne \vec{a}\ne y$, in consequence $x\wedge y\ne \vec{a}$ and 
we can say that $f_{\vec{a},\vec{b}}$ preserve meets.
\eqref{item2} The proof is analogous.
\end{proof}

\begin{subsection}{Complements} 
In \cite{Birk} a \emph{complemented lattice} is defined as a 
bounded lattice (with least element 0 and greatest element 1), 
in which every element $a$ has a \emph{complement}, i.e. an element 
$b$ such that $a \vee {b} = 1$ and $a \wedge {b} = 0$.
Also, given a lattice $L$ and $x\in{L}$ we say that $\hat{x}$ is 
an \emph{orthocomplement of $x$} if the following conditions are satisfied:
\begin{itemize}
\item{$\hat{x}$ is a complement of $x$}
\item{$\hat{\hat{x}}=x$}
\item{if $x\leq{y}$ then $\hat{y}\leq{\hat{x}}$}.
\end{itemize}
A lattice is \emph{orthocomplemented} if every element has an orthocomplement. We give a slightly different approach:

\begin{defn}
We say that an upper bounded lattice $(L,\leq,\wedge,\vee,1)$ with a set of minimal elements $\{m_1,...,m_k\}$ is \emph{semi-orthocomplemented} if every element $a\in {L}$ has a complement, i.e. an element 
$b$ such that $a \vee {b} = 1$ and $a \wedge {b} = m_i$ for a certain $i\in \{1,...,k\}$.
\end{defn}

\begin{prop}
$(\bigcirc{G(k)},\leq,\wedge,\vee,1_{\odot},s_{\pi})$ is a semi-orthocomplemented lattice.
\end{prop}
\begin{proof}
For $x=\oslash G$ consider $\hat{x}=\ominus G$ where
\[
\ominus^{j}=\begin{cases}
\otimes & \textrm{if }\oslash^{j}=\odot\\
\odot & \textrm{if }\oslash^{j}=\otimes\,,
\end{cases}
\]
\end{proof}
\end{subsection}

\begin{subsection}{Ideals into $\bigcirc G(k)$} Now we consider the existence of certain subsets of our lattice in order to show a way to find and organize \emph{autonomous} subsystems into $\bigcirc G(k)$.
\begin{defn}
Given a lattice $(L,\leq,\wedge,\vee)$, $I \subseteq {L}$ is an \emph{ideal} if and only if for every $x,y\in{I}$ it follows that $x\vee{y}\in {I}$.
\end{defn}
It can be also considered an equivalent definition: 
\begin{defn} Given a lattice $(L,\leq)$, $I \subseteq {L}$ is an \emph{ideal} if the following conditions are satisfied:
\begin{itemize}
\item{for every $a\in{I}$ and every $x\in{L}$ such that $x\leq{a}$ then $x\in{I}$}
\item{for every $a,b\in{I}$ there is $c\in{I}$ such that $a,b\leq{c}$}.
\end{itemize}
\end{defn}
One can found these definitions in \cite{Birk}.
\begin{example}
    
    For $k=3$ we can construct the following ideals into $\bigcirc G(3)$:

\begin{itemize}
    \item{$\bigcirc G(3)$ itself is an ideal}
    \item{every subgraph in the form 

    
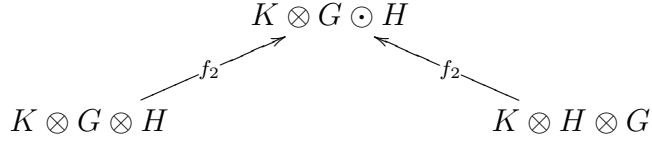
\[
\xymatrix{ & K\otimes G\odot H\\K\otimes G\otimes H\ar[ur]|{f_{2}} & & K\otimes H\otimes G\ar[ul]|{f_{2}} 
}
\]
    is an ideal}
\item{every subgraph in the form 

    
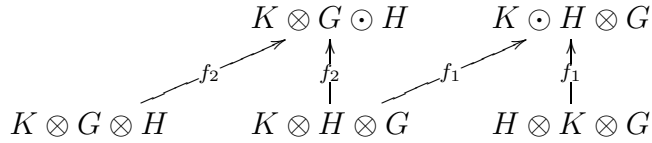
\[
\xymatrix{ & K\otimes G\odot H & K\odot H\otimes G\\K\otimes G\otimes H\ar[ur]|{f_{2}}  & K\otimes H\otimes G\ar[u]|{f_{2}}\ar[ur]|{f_{1}} & H\otimes K\otimes G\ar[u]|{f_{1}}
}
\]
    is an ideal}

\item{...}

\end{itemize}
\end{example}

\end{subsection}

\section{Lattice-ordered partial monoid}\label{monoid}

In this section some concepts from \cite{Jasem} are taken and adapted for the case of a partial operation. Observe that the election of the binary operation is fundamental since it will 
represent the behavior on which we are interested for analyzing.

\begin{defn}
    A system $(A,+,\leq,\wedge,\vee)$ is called a \emph{lattice-ordered partial monoid} if 
    \begin{itemize}
        \item{$(A,+)$ is a partial monoid}
        \item{$(A,\leq)$ is a lattice with $\wedge$ and $\vee$}
        \item{$a\leq{b}$ implies $a+x\leq{b+x}$ and $x+a\leq{x+b}$}
        \item{$a+(b\vee{c})=(a+b)\vee{(a+c)}, (b\vee{c})+a=(b+a)\vee{(c+a)}$} 
        \item{$a+(b\wedge{c})=(a+b)\wedge{(a+c)}, (b\wedge{c})+a=(b+a)\wedge{(c+a)}$}
    \end{itemize}
for every $a,b,c,x\in{A}$. 
\end{defn}

We are introducing a different feature from the operation considered 
in \cite{Jasem} since $+$ defined here is partial, this is oriented to the study 
of $\bigcirc G(k)$ as a lattice-ordered 
partial monoid. For that we need a partial semigroup structure for our set, this is 
obtained by endowing it with the partial operation $+$ defined for 
$x, y \in \bigcirc{G(k)}$ in the form:

\[
x+y=\begin{cases}
y & \textrm{if } x\geq{y}\\
x & \textrm{if }y\geq{x}
\end{cases}
\]

Now $+$ is an associative, commutative and partial operation. 
It is actually a \emph{partial minimum}. The election of this 
operation is due to the idea that the composition of two comparable 
multilayers annihilates the bigger one.

\begin{prop}
$(\bigcirc G(k),+)$ is a partial commutative monoid.
\end{prop}

\begin{proof}
Operation $+$ satisfies associativity: suppose that 
$x,y,z\in \bigcirc G(k)$ are comparable to each other. Now: $$x+y+z=min(x,y,z)=min(min(x,y),z)=min(x,min(y,z))\,.$$
As $\bigcirc G(k)$ is finite, the unique top element (see Lemma \ref{lem1})
is the identity element.
\end{proof}

Let $(M,\cdot)$ be a partial monoid and let $f\colon M\to M$ 
be a mapping. Recall that a \emph{partial homomorphism} between partial 
monoids is mapping that preserves the binary operation, namely 
$f(x+y)=f(x)+f(y)$, $f(1)=1$ and $x+y\in M$ implies that $f(x)+f(y)\in M$.
A mapping between lattice-ordered partial monoids is a \emph{lattice partial 
homomorphism} if it is a partial homomorphism of partial monoids that
preserves meets and joins.
\begin{prop}
Mappings $f_j$ are partial homomorphisms.
\end{prop}
\begin{proof}
By virtue of Proposition \ref{prop2}, mappings $f_j$ preserve meets 
and joins. From definition of the partial operation $+$, we konw that $x$ and $y$ are comparable 
if and only if there exists $x+y$. As $f_j$ is monotone, if $x\le y$, then $f_j(x)\le f_j(y)$
and $f_j(x)$ and $f_j(y)$ are comparable. So 
$f_j(x+y)=f_j(\min(x,y))=f_j(x) = \min(f_j(x),f_j(y))= f_j(x)+f_j(y)$.
Finally as $1$ is the top element $x\le 1$ for every $x$, thus 
$f_j(1)\le 1$. But $f_j$ is closure, hence $1\le f_j(1)$. Therefore $f(1)=1$.
\end{proof}

Let us prove a version for mappings defined by two tuples from Section \ref{interior}.
We notice, as in Proposition \ref{prop1}, that the disjointness hypothesis 
is for simplify the proof. We follow the same notation and definition of 
strictily not absorbing given in the previous section. Also notice that we show
the result for monoids (not partial monoids).
\begin{prop}
Let $(L,+)$ be a monoid and assume that $\{\vec{a}\}\cap\{\vec{b}\}=\emptyset$.
If $\vec{b}$ is an absorbing element and $\vec{a}$ is strictily not absorbing 
for $+$, then mapping $f_{\vec{a},\vec{b}}$ is an homomorphism.
\end{prop}
\begin{proof}
As $\vec{b}$ is absorbing:
\[
f_{\vec{a},\vec{b}}(x)+f_{\vec{a},\vec{b}}(y)=\begin{cases}
\vec{b}+\vec{b}=\vec{b} & \textrm{if } x=y=a\\
x+\vec{b}=\vec{b} & \textrm{if } x\ne a;y=a\\
\vec{b}+y=\vec{b} & \textrm{if } x=a;y\ne a\\
x+y & \textrm{if } x\ne a;y\ne a
\end{cases}
\]
As $\vec{a}$ is strictily not absorbing for $+$:
$f_{\vec{a},\vec{b}}(x)+f_{\vec{a},\vec{b}}(y)=x+y$ if and only 
if $x\ne \vec{a}$ and $y\ne \vec{a}$ if and only if $x+y\ne \vec{a}$
if and only if $f_{\vec{a},\vec{b}}(x+y)=x+y$.
\end{proof}

\begin{prop}
$(\bigcirc{G(k)},+,\leq)$ is a lattice-ordered partial monoid.
\end{prop}

\begin{proof}
    Suppose that $x,y,z\in {\bigcirc{G(k)}}$ are comparable. Observe that $$x+(y\vee{z})=(x+y)\vee{(x+z)}, (y\vee{z})+x=(y+x)\vee{(z+x)}$$ and $$x+(y\wedge{z})=(x+y)\wedge{(x+z)}, (y\wedge{z})+x=(y+x)\wedge{(z+x)}$$
together with the fact that for $x\leq{y}$: 
$$
x+z\leq{y+z}, z+x\leq{z+y}\,.
$$
Notice in particular that
$$
x+(y\vee{z})=min(x,y\vee{z})=\begin{cases}
min(x,\odot) & \textrm{if } y=\odot \text{ or } z=\odot\\
min(x,\otimes) & \textrm{else }
\end{cases}
=\begin{cases}
x & \textrm{if } y=\odot \text{ or } z=\odot\\
\otimes & \textrm{else }
\end{cases}
$$
equals to
$$
(x+y)\vee{(x+z)}=min(x,y)\vee{min(x,z)}=\begin{cases}
\odot & \textrm{if } min(x,y)=\odot \text{ or } min(x,z)=\odot\\
\otimes & \textrm{else }
\end{cases}
$$
$$
=\begin{cases}
\odot & \textrm{if } x=y=\odot \text{ or } x=z=\odot\\
\otimes & \textrm{else}
\end{cases}
$$
\end{proof}

In \cite{Jasem} we found that if for elements $x,y\in{\mathbf{\bigcirc}{G(k)}}$ 
there exist a least $a\in \mathbf{\bigcirc}{G(k)}$ such that $x+a \geq y$, 
then the element $a$ is denoted by $y-x$.
\begin{defn}
A system $(A, +, \le, 0, \wedge, \vee, -)$ is called a \emph{dually residuated lattice
partial monoid} (notation DRl-partial monoid) if
\begin{enumerate}
\item{$(A, +, \le, \wedge, \vee)$ is a lattice ordered partial monoid with $0$;}
\item{for each $x,y\in A$ there exist an element $y-x$;}
\item{$b + ((a - b) \vee 0) \leq a \vee b$, $((a - b) \vee 0) + b \leq a \vee b$ for each $x,y\in A$;}
\item{$x-x\geq 0$ for each $x\in A$.}
\end{enumerate}
\end{defn}

\begin{prop}
{$\mathbf{\bigcirc}{G(k)}$ is a DRl-partial monoid.}    
\end{prop}

\begin{proof}
   For every $x=\ominus{G},y=\obar{G}\in \mathbf{\bigcirc}{G(k)}$ we define the element $y-x=\oslash{G}$ as 
   \[
\oslash^{j}=\begin{cases}
\odot & \textrm{if }\ominus^{j}=\otimes\textrm{ and }\obar^{j}=\odot\\
\otimes & \textrm{else}
\end{cases}
\] 
and prove condition 3. leaving condition 4. as an easy exercise. For every $x=\ominus{G},y=\obar{G}\in \mathbf{\bigcirc}{G(k)}$ we have 

$$\ominus{G} + (\oslash{G} \vee 0)=\ominus{G} + (\oslash{G} \vee 1_{\odot})=\ominus{G} + 1_{\odot}=\ominus{G}\le\ominus{G} \vee \obar{G}$$
\end{proof}

\begin{subsection}{The deletion property} We finish the paper with the \emph{deletion property}, which is studied in \cite{Brown} 
and also apply to our context.
\begin{defn}
    A \emph{left-regular band} is a semigroup $(S,+)$ such that for every $x \in{S}$:
    \begin{itemize}
        \item{$x$ is idempotent}
        \item{$x+y+x=x+y$}
    \end{itemize}
\end{defn}

The second condition is known as \emph{the deletion property} 
(see \cite{Brown}) since it amounts to the fact that we can remove 
from every addition a summand that has appeared earlier without 
changing the value of the addition.

\begin{lem}
  ($\mathbf{\bigcirc}{G(k)},+)$ is a left-regular band.  
\end{lem}

\begin{proof}
    That the deletion property is satisfied in $\mathbf{\bigcirc}{G(k)}$ is straightforward and says essentially 
that $$min(x,y,x)=min(y,x,y)=min(x,y)$$
\end{proof}

Observe that we could have defined the ordering into 
$\mathbf{\bigcirc}{G(k)}$ by means of $$x \leq {y} \text{ if and only if } x+y=y$$
see \cite{Brown}.
\end{subsection}

\section*{Acknowledgment}
We thank the referee for carefully reading and valuable suggestions.

\bibliographystyle{comnet,plainnat}
\bibliography{sampleBibFile}

\begin{thebibliography}{15}

\bibitem{Ben} {J. van'Benthem} (1985)
Symbolic Logics, 
Monographs in Philosophical Logics and Formal Linguistics, 
vol. III, Napoli: Bibliopolis.

\bibitem{Birk} {G. Birkhoff} (1948) Lattice Theory, 
A.M.S. Colloquium Publications, vol. 25, Revised Edition, New York, 1948.


\bibitem{Bocetal} {S. Boccaletti, G. Bianconi, R. Criado, C.I. del Genio,
J. G\'omez-Garde\~nes, M. Romance, J. Sendi\~na-Nadal, Z. Wang, M. Zanin} (2014)
The structure and dynamics of multilayer networks, 
Physical Reports, \textbf{544}(1), pp. 1--122.

\bibitem{Baezetal} {J.C. Baez, J. Foley, J. Moeller, B.S. Pollard} (2020)
Network models, 
Theory Appl. Categ., \textbf{35}, pp. 700--744.

\bibitem{Bliz} {W.D. Blizard} (1991)
The development of multiset theory, 
Mod. Log., \textbf{1}(4), pp. 319--352.

\bibitem{Brown} {K. Brown} (2000)
Semigroups, rings, and Markov chains, 
J. Theoret. Probab., \textbf{13}(3), pp. 871--938.

\bibitem{Cha} {I. Chajda, H. L\"anger} (2023) Monotone and cone preserving mappings on posets, 
Math. Bohem., \textbf{148}(2), pp. 197--210.

\bibitem{Dometal} {M. De Domenico, A. Sol\'e-Ribalta, E. Cozzo, M. Kivel\"a, Y. Moreno, 
M. A. Porter, S. G\'omez \and A. Arenas} (2013)
 Mathematical Formulation of Multilayer Networks
Physical Rev. X \textbf{3} pp. 041022.

\bibitem{Bia} {G. Bianconi} (2022)
 Multilayer networks. Structure and function 
Oxford University Press, Oxford.

\bibitem{Keh} {N. Kehayopulu} (2020)
Lattice ordered semigroups and $\Gamma$-hypersemigroups 
Turkish J. Math. \textbf{44} pp. 1835--1851.

\bibitem{Kivetal} {M. Kivel\"{a}, A. Arenas, M. Barthelemy, J.P. Gleeson, Y. Moreno, M.A. Porter} (2014) Multilayer networks
Journal of complex networks \textbf{2}(3) pp. 203--271.

\bibitem{Jasem} {M. Jasem.} (2003)
 On lattice-ordered monoids
Discussiones Mathematicae, General Algebra and Applications \textbf{23}(2) pp. 101--114.

\bibitem{McLane} {S. Mac Lane} (1998)
 Categories for the Working Mathematician 
Graduate Texts in Mathematics, vol. 5, 2nd Ed., Springer, New York.

\bibitem{Sign} {Camilo Miguel Signorelli, Joaqu\'in D\'iaz Boils, Enzo Tagliazucchi, Bechir Jarraya, Gustavo Deco} (2022)
From brain-body function to conscious interactions,
Neuroscience and Biobehavioral Reviews, Volume 141, 104833, ISSN 0149-7634.

\bibitem{Yao} {Y. Ouyanga, H.-P. Zhangb, Z. Wang, B. De Baets} (2022) On triangular norms representable as ordinal sums based on 
interior operators on a bounded meet semilattice, Fuzzy Sets and Systems 
\textbf{439} pp. 89--101.

\bibitem{Yoe} {M. Yoeli} (1965) Lattice-ordered semigroups, graphs and automata, 
J. Soc. Industrial Appl. Math. \textbf{13}(2) pp. 411--422.


\end{thebibliography}

\end{document}